\documentclass[11pt]{article}
\usepackage{fullpage}
\usepackage{amssymb}
\usepackage{amsmath}
\usepackage{amsthm}
\usepackage{multirow}
\usepackage{enumerate}
\usepackage{graphicx}
\usepackage{mathrsfs}
\usepackage[utf8]{inputenc}
\usepackage{array}
\usepackage{bm}
\usepackage{tikz-cd}
\usepackage{todonotes}
\usepackage{makecell}
\usepackage{hyperref}

\theoremstyle{plain}
\newtheorem{theorem}{Theorem}[section]
\newtheorem{lemma}[theorem]{Lemma}
\newtheorem{proposition}[theorem]{Proposition}

\theoremstyle{definition}
\newtheorem{definition}[theorem]{Definition}

\theoremstyle{remark}

\newcommand{\phunc}{\phi_{\bm{c}}}
\newcommand{\R}{\mathbb{R}}
\newcommand{\Q}{\mathbb{Q}}
\newcommand{\C}{\mathbb{C}}

\newcommand{\N}{\mathbb{N}}
\newcommand{\CP}{\mathbb{C}P}
\newcommand{\RP}{\mathbb{R}P}
\newcommand{\Spt}{\mathcal{S}_+(3,t)}

\newcommand{\cR}{\mathcal{R}}

\newcommand{\cRJ}{\mathcal{R}_J}

\newcommand{\cCJ}{\mathcal{C}_J}

\newcommand{\cS}{\mathcal{S}}

\DeclareMathOperator{\Proj}{Proj}

\makeatletter
\newcommand{\leqnomode}{\tagsleft@true}
\newcommand{\reqnomode}{\tagsleft@false}
\makeatother

\DeclareMathOperator{\val}{val}

\begin{document}

\title{Improved fewnomial upper bounds from Wronskians and dessins d'enfant
}
\date{}
\author{Boulos El-Hilany, Sébastien Tavenas}

\maketitle

\begin{abstract}
We use Grothendieck's dessins d'enfant to show that if $P$ and $Q$ are two real polynomials, any real function of the form $x^\alpha(1-x)^{\beta} P - Q$, has at most $\deg P +\deg Q + 2$ roots in the interval $]0,~1[$. As a consequence, we obtain an upper bound on the number of positive solutions to a real polynomial system $f=g=0$ in two variables where $f$ has three monomials terms, and $g$ has $t$ terms. The approach we adopt for tackling this Fewnomial bound relies on the theory of Wronskians, which was 
used in Koiran et.\ al.\  (J.\ Symb.\ Comput., 2015) for producing the first upper bound which is polynomial in $t$.


\end{abstract}

\section{Introduction}
A classical problem in real algebraic geometry is to find upper bounds on the number of real isolated solutions of a square system of polynomial equations. In numerous instances, systems above describe models in science and engineering, in which positive solutions (i.e., those whose coordinates all have positive values) interpret problems in fields such as control theory~\cite{By89}, kinematics~\cite{BR90}, and chemical reaction networks~\cite{GH02,MFRGCSD16}. Bounding the number of positive solutions in terms of the total number of distinct monomials appearing in the polynomials was a problem considered by Kouchnirenko already in the 1970s~\cite{Ku08}. Such a task can be thought of as one possible multivariate generalization of Descartes' rule of sign which results in an upper bound for the number of positive roots of a univariate polynomial~\cite{Des1637}. In the late 1980s, Khovanskii provided the first bound on the number of positive solutions that depends on the total number of distinct exponent vectors appearing in the system with non-zero coefficients~\cite{Kh91}. Even though his bound was later improved by Bihan and Sottile~\cite{BS07}, the question of sharpness remains open~\cite{BRS08,EH18,MR3892410}. We refer the reader to~\cite{S11} for a broader exposition on the topic.

Several variations of this problem were considered throughout the literature where extra conditions can be imposed on the exponent vectors of the original system~\cite{LRW03,BBS06,bihan2011fewnomial}. In this paper, we consider one such open problem concerning real square polynomial systems of the form
\begin{align}\label{eq:sys1}
 f(x,y)=g(x,y)=0, 
\end{align} where $f$ has $t\geq 3$ non-zero terms and $g$ has three non-zero terms. Li, Rojas, and Wang~\cite{LRW03} published the first upper bound $\Spt\leq 2^t-2$ on the maximal possible number $\Spt$ of positive solutions that any system of the form~\eqref{eq:sys1} can have. A straightforward approach is to substitute one variable of the trinomial in terms of the other, and thus one can reduce~\eqref{eq:sys1} to $F(x) =0$, where $F$ is a smooth function on $]0,1[$ of the form
\begin{align}\label{eq:BigF}
  F(x)= &  \sum_{i=1}^tc_ix^{k_i}(1-x)^{l_i},
\end{align} for which all the coefficients and exponents are real. Then, the number of positive solutions of~\eqref{eq:sys1} is equal to the number of roots of  $F$ contained in $]0,1[$. The main approach used in~\cite{LRW03} is an extension of Rolle's Theorem and a recursion involving derivatives of certain analytic functions. 

The exponential upper bound  above has been reduced by Koiran, Portier, and Tavenas~\cite{KPT15} into a polynomial one. They considered an analytic function in one variable 
\begin{align}\label{eq:Tav}
 \displaystyle\sum_{i=1}^t\prod_{j=1}^m f^{\alpha_{i,j}}_j
\end{align} where all $f_j$ are real polynomials of degree at most $d$ and all the powers of $f_j$ are real. Using the \emph{Wronskian}, it was proved that the number of positive roots of~\eqref{eq:Tav} in an interval $I$ (assuming that $f_j(I)\subset\ ]0,+\infty[$) is bounded by $\frac{t^3md}{3} + 2tmd + t$. As a particular case (taking $m=2$, $d=1$ and $I=\ ]0,1[$), the function~\eqref{eq:Tav} becomes $F$, which consequently has at most $2t^3/3 + 5t$ roots in $]0,1[$.
In fact, in~\cite{KPT15}, the authors only count the number of distinct zeros, but, in the case of the intersection of one trinomial curve with a \(t\)-sparse one, one can obtain the same bound for the count with multiplicity by replacing~\cite[Theorem~9]{KPT15} with~\cite[Theorem~2]{VP75}.

The Wronskian was already used~\cite{VP75,P22,NY95} as a tool to obtain upper bounds on the number of roots of a linear combination of real functions. For example, Polya and Szego~\cite[Part V, Chapter 1, \S 7]{PS76} used it to prove Descartes' sign rule and some generalizations of it. The Wronskian of a given set \(f_1,\ldots,f_k\) of analytic functions is a function of the form
\begin{align*}
W(f_1,\ldots,f_i) = \det \begin{bmatrix}
f_1 & \ldots & f_i \\ f_1^\prime & \ldots & f_i^\prime  \\ \vdots & \vdots & \vdots \\ f_1^{(i-1)} & \ldots & f_i^{(i-1)} \\
\end{bmatrix} .
\end{align*}

The approach in~\cite{KPT15} was very recently refined by M\"uller and Regensburger~\cite{muller2023parametrized}. They obtain the better bound $\Spt\leq  t^3/3 - t^2 + (8/3)t - 2$.
In this paper, we improve again the upper bound on $\Spt$.
\begin{theorem}\label{MAIN.TH}
 $\Spt\leq t^3/3 - 3 t^2/2 + 25 t/6 - 3$.
\end{theorem}

It is straightforward to verify that for every positive integer $t$, our main result provides the sharpest upper bound on $\Spt$ so far (see also Table~\ref{tab:comparisons}). We furthermore recover the upper bound $5$ for $t=3$, discovered in~\cite{LRW03},  which was proven to be sharp by Haas in~\cite{H02}.  Moreover, it is also shown in~\cite{LRW03} that if this bound is reached by a system~\eqref{eq:sys1} of two trinomials, then the Minkowski sum $\Delta$ of the associated Newton polytopes $\Delta_1$ and $\Delta_2$ is a hexagon. 
Later, the first author provided in~\cite[Theorem 4.3]{Hil16} further necessary conditions on $\Delta$ for the above sharpness property on~\eqref{eq:sys1} to be satisfied.

\begin{table}
    \centering
    \begin{tabular}{c||c|c|c|c}
         & ~\cite{LRW03} &~\cite{KPT15} &~\cite{muller2023parametrized} & Theorem~\ref{MAIN.TH} \\[5pt]
        \Xhline{1\arrayrulewidth}
      $t$ & $2^t - 2$ & $2t^3/3 + 5t$ & $t^3/3 - t^2 + 8t/3 - 2$ & $t^3/3 - 3 t^2/2 + 25 t/6 - 3$ \\[5pt]     
      \Xhline{3\arrayrulewidth}

      $3$ & $6$ & $33$ & $6$ & $5$ \\[3pt]     
      $4$ & $14$ & $62$ & $14$ & $11$ \\[3pt]     
      $5$ & $30$ & $108$ & $28$ & $22$ \\[3pt]     
      $6$ & $62$ & $174$ & $50$ & $40$ \\[3pt]  
      \vdots &  &  &  &  \\[3pt]  
      $10$ & $1022$ & $716$ & $258$ & $222$ \\[3pt]  
    \end{tabular}
    \caption{The values are rounded up to the smaller integer}
    \label{tab:comparisons}
\end{table}

 The setup for the proof of Theorem~\ref{MAIN.TH} can be found in~\S\ref{Sec:2}: 
   We study the function $F$ using the same method as in~\cite{KPT15} i.e., we consider a recursion involving derivatives of analytic functions in one variable. In fact, we define a family of 
  Wronskian functions \(R_1,\ldots,R_t\) such that \(R_1 = F\), and the number of zeros of \(R_i\) 
  in \(]0,1[\) is bounded by a function 
  on the number of zeros of \(R_{i+1}\)
  (see Lemma~\ref{L:Recursion}). Unlike the previous approachs (e.g., in~\cite{VP75,KPT15,muller2023parametrized}), we stop at the step \(t-2\) of the recursion to
   obtain the equation 
  \begin{align}
    x^{\alpha}(1-x)^{\beta}P(x) - Q(x) = 0,
  \end{align} 
  where $\alpha,\beta\in\mathbb{R}$, and both $P$ and $Q$ are real polynomials of degree at most $\binom{t-1}{2}$. Instead of completing the usual last step of the recursion (c.f.~\cite{VP75,KPT15,muller2023parametrized}), we show the following result, which gives rise to the upper bound appearing in Theorem~\ref{MAIN.TH}.

\begin{theorem}\label{Main.Th}
Let $P$ and $Q$ be two polynomials in $\R[x]$, and let $\alpha,\beta\in\R$ be any two real values. Then, 
the number of solutions in $]0,1[$ of $x^{\alpha}(1-x)^{\beta}~P-Q$ cannot exceed $\deg P + \deg Q + 2$.
\end{theorem} 

The bulk of this paper (\S\ref{sec:dessins} and~\S\ref{sec:proof_rational}) is devoted to proving Theorem~\ref{Main.Th}. 
Methods we adopt are combinatorial in nature; as a first step, 
% First, in the notation of Theorem~\ref{Main.Th},
one can reduce to studying the number of roots in $]0,1[$ to $\phi - 1$, where $\phi:=x^{\alpha}(1-x)^{\beta}~P/Q$. In fact, we show that there is no loss of generality if one assumes that the exponents $\alpha$ and $\beta$ are rational and satisfy some properties on the parity of the numerators and denominators (this is Proposition~\ref{prp:analytic_to_rational} in~\S\ref{sec:proof_rational}). 
  Then, choosing $m\in\mathbb{N}$ such that both $m\alpha$ and $m\beta$ are integers, we get a complex meromorphic function $\varphi:=\phi^m$ which extends projectively to $\mathbb{C}P^1\longrightarrow \mathbb{C}P^1$. The inverse images of $0$, $\infty$, $1$ are given by the roots of $P$, $Q$, $\varphi - 1$, respectively, together with $0$ and $1$ (if $\alpha\beta\neq 0$). These inverse images lie on the graph $\Gamma:=\varphi^{-1}(\mathbb{R}P^1)\subset \mathbb{C}P^1$, which is an example of a Grothendieck's real \textit{dessin d'enfant} (c.f.~\cite{B07,Br06} and~\cite{O03}).
  
Combinatorics of $\Gamma$ reflect the analytic properties of $\varphi$. Namely, critical points of $\varphi$ correspond to nodes of $\Gamma$, and the number of roots of $\varphi -1$ in $]0,1[$ is controlled by the positions, and relative adjacency relations, of a certain type of vertices in $\Gamma$. The latter correspond to either critical points of $\varphi$, its roots, or its poles.~\S\ref{sec:dessins} is devoted for describing those dessins d'enfant and laying out the necessary groundwork.  In~\S\ref{sec:proof_rational}
we make an analysis on $\Gamma$ to obtain a bound on the number of those vertices that are pertinent to the existence of roots of $\varphi - 1$ in $]0,1[$. The obtained bounds are ultimately expressed in terms of $\deg P$ and $\deg Q$.

 We believe that Theorem~\ref{Main.Th} could be proved using purely analytic methods. Due to the unexpected difficulty of finding such an approach, we decided to use combinatorial tools instead.

\subsubsection*{Acknowledgements}
The authors are grateful to Fr\'ed\'eric Bihan for fruitful discussions.

\section{Proof of Theorem~\ref{MAIN.TH}}\label{Sec:2}
 Define the polynomials $f$ and $g$ of the initial system~\eqref{eq:sys1} as

\begin{align}\label{Main_Poly}  
f(u,v)=\overset{t}{\underset{i=1}{\sum}} a_iu^{\alpha_i}v^{\beta_i} \text{\quad and\quad} g(u,v)=\overset{3}{\underset{j=1}{\sum}} b_ju^{\gamma_j}v^{\delta_j},
\end{align} where all $a_i$ and $b_i$ are real.
 We suppose that~\eqref{eq:sys1} has positive solutions, thus the coefficients of $g$ have different signs. Therefore without loss of generality, let $b_1=-1$, $b_2>0$ and $b_3>0$. Since we are looking for positive solutions with non-zero coordinates, one can assume that $\gamma_1=\delta_1=0$. Theorem~\ref{MAIN.TH} holds true if the vectors $(\gamma_2,\delta_2)$ and $(\gamma_3,\delta_3)$ are collinear. Indeed, it was proved in~\cite{KPT15} that in this case, the bound on the number of positive solutions is either infinite or bounded by $2t-2$. Therefore, we assume in the rest of this section that the Newton polytope of $g$ is a triangle.

The monomial change of coordinates $(u,v)\rightarrow (x,y)$ of $(\mathbb{R}^*)^2$ defined by $b_2u^{\gamma_2}v^{\delta_2}=x$ and $b_3u^{\gamma_3}v^{\delta_3}=y$ preserves the number of positive solutions. Therefore, we are reduced to a system 

\begin{align} \label{Main.poly.modif}
 \overset{t}{\underset{i=1}{\sum}} c_ix^{k_i}y^{l_i}  =  -1  + x + y  =  0
\end{align} where $c_i$ is real for $i=1,\ldots,t$, and all $k_i$ and $l_i$ are rational numbers.

We now look for the positive solutions of~\eqref{Main.poly.modif}. It is clear that since both $x$ and $y$ are positive, then $x\in \ ]0,1[$. Substituting $1-x$ for $y$ in~\eqref{Main.poly.modif}, we get
the function $F$ in~\eqref{eq:BigF}, expressed as $\sum_{i=1}^t f_i$, where $f_i(x):= c_ix^{k_i}(1-x)^{l_i}$. 
Hence the number of positive solutions of~\eqref{eq:sys1} is equal to that of roots of $F$ in $]0,1[$. We can assume that the family \((f_1,\ldots,f_t)\) is linearly independent (if it is not the case, we can rewrite \(F\) as a linear combination of the generators of the family).

Now let us define two families of Wronskian determinants. For \(1\leq j\leq t\), 
\(
    W_j := W(f_1,f_2,\ldots,f_j)
\)
and
\(
    R_j := W(f_1,\ldots,f_{j-1},F)
    % \\
    % & = W(f_1,\ldots,f_{i-1},f_i) + W(f_1,\ldots,f_{i-1},f_{i+1}) +\ldots %+W(f_1,\ldots,f_{i-1},f_{k-1})
    % +W(f_1,\ldots,f_{i-1},f_k).
\).
We also set \(W_0 := 1\).
Let us remark that 
\begin{equation*}
    R_1 = F \textrm{ and } R_t=W(f_1,\ldots,f_t).
\end{equation*} 
All these Wronskian determinants are not identically zero, since the family \((f_1,\ldots,f_t)\) is linearly independent.

In what follows, we give an explicit description of  Wronskians \(W_j\) (this description also appears in Proposition 13 of~\cite{muller2023parametrized}).
Notice that by induction on \(m \ge 0\), \( f_i^{(m)} = c_i x^{k_i-m}(1-x)^{l_i-m} p_{i,m}(x)\) where \(p_{i,m}\) is a non-zero polynomial of degree \(m\). Consequently, the \(i\)-th column of \(W_j\) can be factorized by \(x^{k_i-j+1}(1-x)^{l_i-j+1}\) and then, the \((m+1)\)-th row can be factorized by \((x(1-x))^{j-m-1}\). Hence, we obtain
\[
    W_j = x^{(\sum_1^j k_i)-\frac{j(j-1)}{2}} \cdot (1-x)^{(\sum_1^j l_i)-\frac{j(j-1)}{2}} \cdot P_{(1,\ldots,j)}(x),
\]
where \(P_{\tau}\) (\(\tau\) being a \(j\)-tuple) is the polynomial 
\[
    P_{\tau} = \det \begin{bmatrix}
        c_{\tau_i} \cdot p_{\tau_i,m}(x)
    \end{bmatrix}_{1\le i\le j, 0\le m\le j-1}
\]
of degree at most \(\binom{j}{2}\).

Now, for $j=1,\ldots,t$, let $\mathcal{R}_j$ (resp. $\mathcal{W}_j$) denote the number of roots, counted with multiplicities, of $R_j$ (resp. $W_j$) in $]0,1[$, and set $\mathcal{W}_0=0$. The numbers \(\mathcal{R}_j\)'s and \(\mathcal{W}_j\)'s are finite since the corresponding Wronskians are not identically zero. Note that $\mathcal{R}_1$ is \(\mathcal{S}_{+}(3,t)\), the number of real roots of $F$ in $]0,1[$ (for some choice of parameters \(c_i\), \(k_i\), and \(l_i\)). The previous paragraph implies \(\mathcal{W}_j \le \binom{j}{2}\).

Moreover, we have that
\begin{align*}
    R_{t-1} & = W(f_1,\ldots,f_{t-2},f_{t-1})+ W(f_1,\ldots,f_{t-2},f_{t}) \\
    & = x^{\gamma_1} (1-x)^{\delta_1} P_{(1,\ldots,t-2,t-1)} 
    + x^{\gamma_2} (1-x)^{\delta_2} P_{(1,\ldots,t-2,t)}
\end{align*}
for some real numbers \(\gamma_1,\gamma_2,\delta_1,\delta_2\). By Theorem~\ref{Main.Th}, \(R_{t-1}\) has at most \(2\binom{t-1}{2}+2\) roots on \(]0,1[\), i.e., \(\mathcal{R}_{t-1} \le (t-1)(t-2)+2\).

The next result appears in~\cite[Equation~(6)]{VP75}.
\begin{lemma}\label{L:Recursion}
For $j=1,\ldots,t-1$, we have \begin{equation}\label{eq:Recursion}
\mathcal{R}_j\leq \mathcal{R}_{j+1} + \mathcal{W}_j + \mathcal{W}_{j-1} +1.
\end{equation} 
\end{lemma}

It follows that 
\begin{align*}%\label{eq:AfterRecursion}
    \mathcal{S}_{+}(3,t) = \mathcal{R}_1 
    & \leq \mathcal{R}_{t-1} + \mathcal{W}_{t-2} + 2\cdot \sum_{j=1}^{t-3}\mathcal{W}_j + t-2 \\
    & \leq (t-1)(t-2) + 2 + \binom{t-2}{2} + 2 \cdot \binom{t-2}{3} + t-2 \\ 
    & = \frac{t^3}{3} - \frac{3t^2}{2} + \frac{25t}{6}-3. 
\end{align*}

\section{Dessins d'enfant} \label{sec:dessins}
In this section, we introduce graphs on the Riemann sphere that completely describe rational functions up to homeomorphisms. These graphs  will be key to proving Theorem~\ref{Main.Th} later in~\S\ref{sec:proof_rational}.

 For any two coprime univariate real polynomials $A,B\in\R[x]$, consider the rational function $\varphi:=A/B\in\R(x)$.
Here and in the rest of the paper, we see the source and target spaces of $\varphi:\C\longrightarrow\C$ as the affine charts of $ \mathbb{C}P^1$ given by the non-vanishing of the first coordinate, $x_0$, of the homogeneous coordinates $(x_0:x_1)$. Furthermore, we abuse notations by allocating the same symbol $\varphi$ to denote the rational function $\mathbb{C}P^1\longrightarrow\mathbb{C}P^1$ obtained as a result of homogenization with respect to $(x_0:x_1)$.

The notations we use are taken from~\cite{B07} and for other details, see~\cite{Br06,O03,IZ18} for instance. Define $\Gamma:=\varphi^{-1}(\mathbb{R}P^1)$. 
This is a real graph on $\mathbb{C}P^1$ (invariant with respect to the complex conjugation) and which contains $\mathbb{R}P^1$. Any connected component of $\mathbb{C}P^1\setminus\Gamma$ is homeomorphic to an open disk. Consequently, each \emph{node}  (i.e. a point of valency at least \(3\)) of $\Gamma$ has even valency. Furthermore, if $a\in\CP^1$ is a critical point of $\varphi$ with real critical value (i.e. $\varphi(a)\in\R$) having multiplicity $\mu_a$, then the valency $\val(a)$ of $a$ (as a  node in $\Gamma$)  satisfies
\[
\val(a) = 2(\mu_a + 1).
\]  
Here, we say that $\varphi$ has a critical point $a\in\C$ if $\varphi'(a)=0$. The \emph{multiplicity} of $a$ is defined as the value $\mu \in\N$ such that 
\[
\varphi'(a) = \cdots = \varphi^{(\mu)}(a) = 0,
\] and $\varphi^{(\mu+1)}(a)\neq 0$. 

The graph $\Gamma$ contains the inverse images of $0$, $\infty$, and $1$. 
A {\em vertex} of \(\Gamma\) is any point in the set: 
\[
    V(\Gamma) = \{0,1,\infty\} \cup \varphi^{-1}(\{0,1,\infty\}) \cup \{ \text{nodes of }\Gamma\}.
\] 
Denote by the same letter $p$ (resp. $q$ and $r$) the vertices of $\Gamma$ which are mapped to $0$ (resp. $\infty$ and $1$). Orient the real axis on the target space via the arrows $0\rightarrow \infty\rightarrow 1\rightarrow 0$ (orientation given by the decreasing order in $\mathbb{R}$) and pull back this orientation by $\varphi$. The graph $\Gamma$ becomes an oriented graph, with the orientation given by arrows $p\rightarrow q\rightarrow r\rightarrow p$. The graph $\Gamma$ is called {\it real dessin d'enfant} associated to $\varphi$. In what follows, we will use the term \emph{dessin} for short, and \emph{real dessin} when emphasising that $\Gamma$ is invariant under the action of complex conjugation. 

A \emph{cycle} of $\Gamma$ is the boundary of a connected component of $\mathbb{C}P^1\setminus\Gamma$. Any such cycle contains the same non-zero number of letters $r$, $p$, $q$ with orientation induced from that of $\Gamma$ (see Figure~\ref{fleches}). We say that a cycle obeys the \emph{cycle rule}.

\begin{figure}[h]
\centering
\includegraphics[scale=1.5]{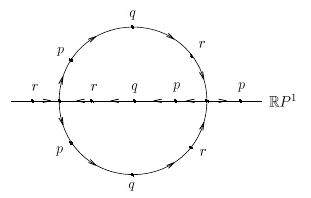}
\caption{Cycles of a dessin.}\label{fleches}
\end{figure}

Conversely, suppose we are given a real connected oriented graph $\Gamma\subset\mathbb{C}P^1$ that contains $\mathbb{R}P^1$, symmetric with respect to $\mathbb{R}P^1$ and having vertices of even valency, together with a real continuous map $\psi:\Gamma\longrightarrow\mathbb{R}P^1$. Suppose that 
each cycle 
contains the same non-zero number of vertices labeled $p$, $q$ and $r$ satisfying the cycle rule. Choose coordinates on the target space $\C P^1$. Choose one connected component of  $\C P^1 \setminus \Gamma$ and send it holomorphically to one connected component of $\C P^1 \setminus \R P^1$ so that vertices with labels $p$ are sent to $(1:0)$, those with label $q$ are sent to $(0:1)$ and those with label $r$  to $(1:1)$. Do the same for each connected component of $\C P^1 \setminus \Gamma$ so that the resulting homeomorphisms extend to an orientation preserving continuous map $\psi: \C P^1 \rightarrow \C P^1$. Note that two adjacent connected components of $\C P^1 \setminus \Gamma$ are sent to different connected components of $\C P^1 \setminus \R P^1$. The Riemann Uniformization Theorem implies that $\psi$ is a real rational map for the standard complex structure on the target space and its pull-back by $\psi$ on the source space.

 Since the graph is invariant under complex conjugation, it is determined by its intersection with one connected component $H$ (for half) of $\mathbb{C}P^1\setminus\mathbb{R}P^1$. In some figures we will only show one half part $H\cap\Gamma$ together with $\mathbb{R}P^1=\partial H$ represented as a horizontal line.

\begin{definition}\label{Def_Neighbors}

    A vertex $v$ is called \emph{special} if it is labeled by either the letter $p$ or $q$. Otherwise, it is called \emph{non-special}. Furthermore, we call $v$ \emph{real} if it belongs to $\R P^1$. Finally, we say that two nodes $a$ and $b$ are \emph{neighbors} if there is a branch of $\Gamma\setminus\mathbb{R}P^1$ joining them such that this branch does not contain any vertices of $\Gamma$ other than $a$ or $b$.
\end{definition}

\subsection{Simple dessins}

In this section, we work with a particular family of dessins.

\begin{definition}\label{def:simple_pair}
A dessin $\Gamma\subset\C P^1$ is called \emph{simple} if it satisfies the following two properties
\begin{itemize}
    \item[\textbf{(i)}] Each non-special node has valency exactly four and is not labeled by \(r\).
    \item[\textbf{(ii)}] All \textit{real} non-special nodes of \(\Gamma\) are neighbors to real nodes.
\end{itemize}
Moreover if \(J\) is a strict open interval of \(\RP^1\), then a simple dessin \(\Gamma\) is called \emph{simple with respect to }\(J\) if
\begin{itemize}   
    \item[\textbf{(iii)}]  All special vertices of $\Gamma\setminus(\RP^1\setminus J)$ have the same valency $2m$ (for some integer \(m\)).
\end{itemize}

\end{definition} 

 We introduce a sequence of transformations that deforms any dessin $\Gamma$ to a simple dessin $\tilde{\Gamma}$. Moreover, if \(\Gamma\) is the dessin of a rational function whose zeros and poles of \(\CP^1 \setminus\{0,1,\infty\}\) have multiplicity a multiple of \(m\), then the new dessin \(\tilde{\Gamma}\) is simple with respect to the interval \(]0,1[\). Finally, each transformation does not reduce the number of \textit{real} vertices labeled by $r$ of $\Gamma$, and even of \(]0,1[\) (c.f.\ the next two subsections). Therefore, to prove Theorem~\ref{Main.Th}, it suffices to consider a simple dessin with respect to \(]0,1[\).

 We divide the set of transformations into two types. The first type, called type \textbf{(I)}, reduces the valencies of all critical points so they verify the conditions \textbf{(i)} and \textbf{(iii)}. The second type, called type \textbf{(II)}, turns a couple of conjugate vertices $(p,\overline{p})$ (resp. $(q,\overline{q})$, $(r,\overline{r})$, non-special critical points) into one vertex $\tilde{p}$ (resp. $\tilde{q}$, $\tilde{r}$, non-special critical point) which is real.

\subsubsection*{Transformation of type (I)}
Consider a node \(\alpha\) of \(\Gamma\) which does \textit{not} belong to $\{0,1,\infty\}$. The following transformation on $\alpha$ will be simultaneously performed on $\overline{\alpha}$ whenever $\alpha\notin\mathbb{R}P^1$.

Let $\mathcal{U}_\alpha$ be a small neighborhood of $\alpha$ in $\mathbb{C}P^1$ such that $\mathcal{U}_{\alpha}\setminus \{\alpha\}$ does not contain any vertices.

 Assume that $\alpha$ is a special node of valency $2km$ for some integer $k \ge 2$. We transform $\Gamma$ inside $\mathcal{U}_\alpha$ as in Figure~\ref{fig:TypeIa}. In the new graph $\Gamma'$, the neighborhood $\mathcal{U}_\alpha$ contains two real special vertices and a real non-special node (and no other vertices). If $\alpha$ is labeled by \(p\) (resp. \(q\)) then both special points are labeled by \(p\) (resp. \(q\)) with valencies $2m$ and $2(k-1)m$. Moreover, the new non-special node has valency $4$. It is obvious that the resulting graph $\Gamma'$ is still a real dessin. We can notice that the number of branches at $\partial \mathcal{U}_\alpha $ is unchanged $(2m-1) + (4 - 2) + (2(k-1)m-1) = 2km$. 
 Assume that $\alpha$ is a non-special node labeled by $r$ of valency $2k$ with $k\geq 2$. We transform the graph $\Gamma$ as in Figure~\ref{Deform_real_r}. In the new graph $\Gamma'$, the neighborhood $\mathcal{U}_{\alpha}$ contains two vertices labeled by $r$ of valency $2(k-1)$ and $2$ respectively, and one non-special node of valency $4$.

 Assume that $\alpha$ is a non-special node that is \textit{not} labeled by $r$ and of valency $2k$ with $k\geq 3$. We transform the graph $\Gamma$ such that in the new graph $\Gamma'$, the neighborhood $\mathcal{U}_{\alpha}$ contains two non-special nodes of valency $4$ and $2(k-1)$.

\begin{figure}
    \centering
    \tikzset{every picture/.style={line width=0.75pt}} %set default line width to 0.75pt        

\begin{tikzpicture}[x=0.75pt,y=0.75pt,yscale=-1,xscale=1]
%uncomment if require: \path (0,292); %set diagram left start at 0, and has height of 292

%Curve Lines [id:da47413761838395785] 
\draw [line width=1.5]    (153.87,153.53) .. controls (165.77,176.62) and (203.11,171.41) .. (219.33,170.97) ;
%Curve Lines [id:da48982133998678423] 
\draw [line width=1.5]    (190.13,236.97) .. controls (150.93,223.77) and (141.73,173.77) .. (153.87,153.53) ;
%Curve Lines [id:da4364003576343667] 
\draw [line width=1.5]    (153.87,153.53) .. controls (152.93,189.77) and (196.53,187.37) .. (215.73,187.37) ;
%Curve Lines [id:da8797183900605617] 
\draw [line width=1.5]    (153.72,252.07) .. controls (130.53,222.57) and (138.13,166.57) .. (153.72,153.86) ;
%Curve Lines [id:da91834998111409] 
\draw [line width=1.5]    (88.93,171.37) .. controls (120.53,176.57) and (138.93,164.17) .. (153.72,153.56) ;
%Curve Lines [id:da7495669154941229] 
\draw [line width=1.5]    (93.73,196.17) .. controls (120.13,186.57) and (135.33,172.17) .. (153.72,153.86) ;

%Curve Lines [id:da777530469697079] 
\draw [line width=1.5]    (153.67,153.73) .. controls (165.57,130.64) and (202.91,135.85) .. (219.13,136.3) ;
%Curve Lines [id:da7502972520016968] 
\draw [line width=1.5]    (189.93,70.3) .. controls (150.73,83.5) and (141.53,133.5) .. (153.67,153.73) ;
%Curve Lines [id:da15821863364752686] 
\draw [line width=1.5]    (153.67,153.73) .. controls (152.73,117.5) and (196.33,119.9) .. (215.53,119.9) ;
%Curve Lines [id:da06919971252892299] 
\draw [line width=1.5]    (153.52,55.2) .. controls (130.33,84.7) and (137.93,140.7) .. (153.52,153.4) ;
%Curve Lines [id:da8152826034527942] 
\draw [line width=1.5]    (88.73,135.9) .. controls (120.33,130.7) and (138.73,143.1) .. (153.52,153.71) ;
%Curve Lines [id:da9912331764355872] 
\draw [line width=1.5]    (93.53,111.1) .. controls (119.93,120.7) and (135.13,135.1) .. (153.52,153.4) ;

%Shape: Ellipse [id:dp7497253952996974] 
\draw  [color={rgb, 255:red, 155; green, 155; blue, 155 }  ,draw opacity=1 ] (87,153.86) .. controls (87,99.62) and (116.87,55.66) .. (153.72,55.66) .. controls (190.57,55.66) and (220.44,99.62) .. (220.44,153.86) .. controls (220.44,208.1) and (190.57,252.07) .. (153.72,252.07) .. controls (116.87,252.07) and (87,208.1) .. (87,153.86) -- cycle ;
%Straight Lines [id:da8946309295441673] 
\draw [line width=1.5]    (87,153.26) -- (220.44,153.86) ;
%Shape: Circle [id:dp45283889956694434] 
\draw  [fill={rgb, 255:red, 0; green, 0; blue, 0 }  ,fill opacity=1 ] (149.22,153.56) .. controls (149.22,151.63) and (150.79,150.06) .. (152.72,150.06) .. controls (154.66,150.06) and (156.22,151.63) .. (156.22,153.56) .. controls (156.22,155.49) and (154.66,157.06) .. (152.72,157.06) .. controls (150.79,157.06) and (149.22,155.49) .. (149.22,153.56) -- cycle ;
%Curve Lines [id:da4928489323086557] 
\draw [color={rgb, 255:red, 155; green, 155; blue, 155 }  ,draw opacity=1 ]   (244.44,142.74) .. controls (274.17,143) and (269.96,161.32) .. (299.27,162.68) ;
\draw [shift={(301.11,162.74)}, rotate = 181.22] [fill={rgb, 255:red, 155; green, 155; blue, 155 }  ,fill opacity=1 ][line width=0.08]  [draw opacity=0] (12,-3) -- (0,0) -- (12,3) -- cycle    ;
%Shape: Ellipse [id:dp6887735028914611] 
\draw  [color={rgb, 255:red, 155; green, 155; blue, 155 }  ,draw opacity=1 ] (322,153.86) .. controls (322,99.62) and (351.87,55.66) .. (388.72,55.66) .. controls (425.57,55.66) and (455.44,99.62) .. (455.44,153.86) .. controls (455.44,208.1) and (425.57,252.07) .. (388.72,252.07) .. controls (351.87,252.07) and (322,208.1) .. (322,153.86) -- cycle ;
%Straight Lines [id:da4829401336752144] 
\draw [line width=1.5]    (322,153.26) -- (455.44,153.86) ;
%Shape: Circle [id:dp8783606096335522] 
\draw  [fill={rgb, 255:red, 0; green, 0; blue, 0 }  ,fill opacity=1 ] (403.83,152.58) .. controls (403.83,150.64) and (405.4,149.08) .. (407.33,149.08) .. controls (409.27,149.08) and (410.83,150.64) .. (410.83,152.58) .. controls (410.83,154.51) and (409.27,156.08) .. (407.33,156.08) .. controls (405.4,156.08) and (403.83,154.51) .. (403.83,152.58) -- cycle ;
%Curve Lines [id:da7280693893746801] 
\draw [line width=1.5]    (407.33,152.58) .. controls (399.73,179.9) and (393.73,209.1) .. (421.33,239.5) ;
%Curve Lines [id:da27841406401726543] 
\draw [line width=1.5]    (407.33,152.58) .. controls (384.53,186.7) and (370.93,220.7) .. (398.13,251.1) ;
%Curve Lines [id:da3444878693543718] 
\draw [line width=1.5]    (325.87,186.7) .. controls (337.87,179.1) and (349.87,170.3) .. (361.07,153.5) ;
%Curve Lines [id:da5005781140722111] 
\draw [line width=1.5]    (334.27,207.9) .. controls (350.67,194.3) and (363.87,172.7) .. (361.07,153.5) ;
%Curve Lines [id:da4544612560999589] 
\draw [line width=1.5]    (407.33,152.58) .. controls (416.53,171.38) and (436.27,180.3) .. (451.87,185.5) ;
%Curve Lines [id:da5988708882608901] 
\draw [line width=1.5]    (444.27,206.3) .. controls (429.07,200.3) and (417.87,189.9) .. (407.33,152.58) ;

%Shape: Circle [id:dp2274833132913916] 
\draw  [fill={rgb, 255:red, 0; green, 0; blue, 0 }  ,fill opacity=1 ] (358.25,153.47) .. controls (358.25,151.95) and (359.48,150.72) .. (361,150.72) .. controls (362.52,150.72) and (363.75,151.95) .. (363.75,153.47) .. controls (363.75,154.99) and (362.52,156.22) .. (361,156.22) .. controls (359.48,156.22) and (358.25,154.99) .. (358.25,153.47) -- cycle ;
%Curve Lines [id:da5924826746106372] 
\draw [line width=1.5]    (407.33,154.1) .. controls (399.73,126.78) and (393.73,97.58) .. (421.33,67.18) ;
%Curve Lines [id:da2522232611634744] 
\draw [line width=1.5]    (407.33,154.1) .. controls (384.53,119.98) and (370.93,85.98) .. (398.13,55.58) ;
%Curve Lines [id:da7629416152841629] 
\draw [line width=1.5]    (325.87,119.98) .. controls (337.87,127.58) and (349.87,136.38) .. (361.07,153.18) ;
%Curve Lines [id:da4460948720798319] 
\draw [line width=1.5]    (334.27,98.78) .. controls (350.67,112.38) and (363.87,133.98) .. (361.07,153.18) ;
%Curve Lines [id:da8092145135160319] 
\draw [line width=1.5]    (407.33,154.1) .. controls (416.53,135.3) and (436.27,126.38) .. (451.87,121.18) ;
%Curve Lines [id:da4764876501143195] 
\draw [line width=1.5]    (444.27,100.38) .. controls (429.07,106.38) and (417.87,116.78) .. (407.33,154.1) ;

%Curve Lines [id:da4891166932521831] 
\draw [line width=0.75]    (114.57,75.06) .. controls (127.57,84.56) and (125.67,123.88) .. (152.72,153.56) ;
%Curve Lines [id:da3024409262406008] 
\draw [line width=0.75]    (201.17,84.88) .. controls (182.67,94.38) and (152.17,99.88) .. (152.72,153.56) ;

%Curve Lines [id:da22270103276214948] 
\draw [line width=0.75]    (113.86,232.94) .. controls (126.86,223.44) and (124.95,184.12) .. (152.01,154.44) ;
%Curve Lines [id:da11284269017721671] 
\draw [line width=0.75]    (200.45,223.12) .. controls (181.95,213.62) and (151.45,208.12) .. (152.01,154.44) ;

%Curve Lines [id:da9426777458781708] 
\draw [line width=0.75]    (354.57,70.45) .. controls (367.57,79.95) and (378.5,110.97) .. (384.22,153.81) ;
%Curve Lines [id:da06272312667593971] 
\draw [line width=0.75]    (434.33,83.47) .. controls (422.17,91.97) and (401.61,109.99) .. (407.33,152.83) ;

%Curve Lines [id:da38995298400469924] 
\draw [line width=0.75]    (354.43,236.37) .. controls (367.43,226.87) and (378.36,195.86) .. (384.08,153.01) ;
%Curve Lines [id:da7476718542629949] 
\draw [line width=0.75]    (434.19,223.36) .. controls (422.02,214.86) and (401.47,196.84) .. (407.19,153.99) ;

% Text Node
\draw (38,145.4) node [anchor=north west][inner sep=0.75pt]  [font=\small]  {$RP1$};
% Text Node
\draw (428.04,146.03) node [anchor=south] [inner sep=0.75pt]  [font=\scriptsize]  {$p$};
% Text Node
\draw (372.8,146.23) node [anchor=south] [inner sep=0.75pt]  [font=\scriptsize]  {$p$};

\end{tikzpicture}
    \caption{A transformation of type \textbf{(I)} where $\alpha$ is a real vertex labeled by $p$, $k=3$ and $m=3$.}
    \label{fig:TypeIa}
\end{figure}
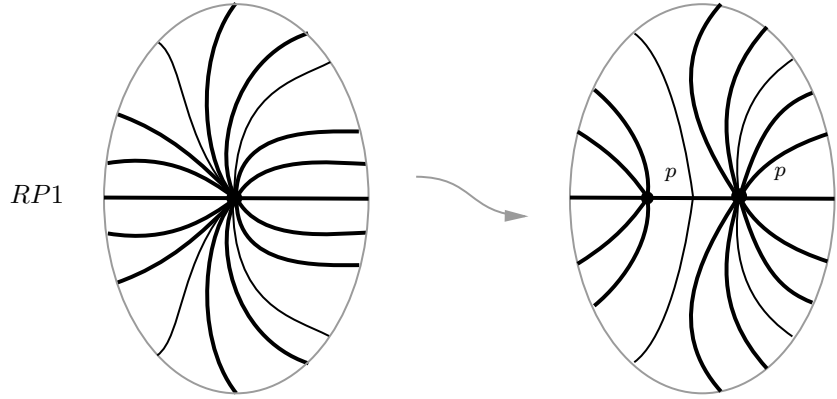

\begin{figure}
    \centering
    \tikzset{every picture/.style={line width=0.75pt}} %set default line width to 0.75pt        

\begin{tikzpicture}[x=0.75pt,y=0.75pt,yscale=-1,xscale=1]
%uncomment if require: \path (0,292); %set diagram left start at 0, and has height of 292

%Shape: Ellipse [id:dp01692450240236787] 
\draw  [color={rgb, 255:red, 155; green, 155; blue, 155 }  ,draw opacity=1 ] (87,153.86) .. controls (87,99.62) and (116.87,55.66) .. (153.72,55.66) .. controls (190.57,55.66) and (220.44,99.62) .. (220.44,153.86) .. controls (220.44,208.1) and (190.57,252.07) .. (153.72,252.07) .. controls (116.87,252.07) and (87,208.1) .. (87,153.86) -- cycle ;
%Straight Lines [id:da25586424517318374] 
\draw [line width=1.5]    (87,153.26) -- (220.44,153.86) ;
%Shape: Circle [id:dp26223427799690924] 
\draw  [fill={rgb, 255:red, 0; green, 0; blue, 0 }  ,fill opacity=1 ] (150.22,153.56) .. controls (150.22,151.63) and (151.79,150.06) .. (153.72,150.06) .. controls (155.66,150.06) and (157.22,151.63) .. (157.22,153.56) .. controls (157.22,155.49) and (155.66,157.06) .. (153.72,157.06) .. controls (151.79,157.06) and (150.22,155.49) .. (150.22,153.56) -- cycle ;
%Curve Lines [id:da034732518201990414] 
\draw [line width=1.5]    (153.87,153.53) .. controls (165.77,176.62) and (201.11,183.13) .. (217.33,182.68) ;
%Curve Lines [id:da7070082170375906] 
\draw [line width=1.5]    (89.33,175.54) .. controls (108.22,165.39) and (137.64,153.98) .. (153.87,153.53) ;
%Curve Lines [id:da158334123415375] 
\draw [line width=1.5]    (153.87,153.53) .. controls (140.44,174.7) and (110.57,198.17) .. (101.9,215.06) ;

%Curve Lines [id:da8563965810132249] 
\draw [color={rgb, 255:red, 155; green, 155; blue, 155 }  ,draw opacity=1 ]   (244.44,142.74) .. controls (274.17,143) and (269.96,161.32) .. (299.27,162.68) ;
\draw [shift={(301.11,162.74)}, rotate = 181.22] [fill={rgb, 255:red, 155; green, 155; blue, 155 }  ,fill opacity=1 ][line width=0.08]  [draw opacity=0] (12,-3) -- (0,0) -- (12,3) -- cycle    ;
%Curve Lines [id:da8350780614452179] 
\draw [line width=1.5]    (153.72,153.01) .. controls (165.63,129.92) and (200.97,123.42) .. (217.19,123.87) ;
%Curve Lines [id:da29072709446449707] 
\draw [line width=1.5]    (89.19,131.01) .. controls (108.08,141.16) and (137.5,152.57) .. (153.72,153.01) ;
%Curve Lines [id:da8792171007250826] 
\draw [line width=1.5]    (153.72,153.01) .. controls (140.3,131.85) and (110.43,108.37) .. (101.76,91.48) ;

%Shape: Ellipse [id:dp6293418795393573] 
\draw  [color={rgb, 255:red, 155; green, 155; blue, 155 }  ,draw opacity=1 ] (322,153.86) .. controls (322,99.62) and (351.87,55.66) .. (388.72,55.66) .. controls (425.57,55.66) and (455.44,99.62) .. (455.44,153.86) .. controls (455.44,208.1) and (425.57,252.07) .. (388.72,252.07) .. controls (351.87,252.07) and (322,208.1) .. (322,153.86) -- cycle ;
%Straight Lines [id:da06335235534382122] 
\draw [line width=1.5]    (322,153.26) -- (455.44,153.86) ;
%Shape: Circle [id:dp740385508764607] 
\draw  [fill={rgb, 255:red, 0; green, 0; blue, 0 }  ,fill opacity=1 ] (403.83,152.58) .. controls (403.83,150.64) and (405.4,149.08) .. (407.33,149.08) .. controls (409.27,149.08) and (410.83,150.64) .. (410.83,152.58) .. controls (410.83,154.51) and (409.27,156.08) .. (407.33,156.08) .. controls (405.4,156.08) and (403.83,154.51) .. (403.83,152.58) -- cycle ;
%Curve Lines [id:da02722848782469045] 
\draw [line width=1.5]    (407.33,152.58) .. controls (411.13,180.96) and (418.86,198.27) .. (438.19,219.6) ;
%Curve Lines [id:da4987905478069842] 
\draw [line width=1.5]    (407.33,152.58) .. controls (397.13,210.62) and (370.33,223.58) .. (345.33,227.91) ;

%Shape: Circle [id:dp5766814507922854] 
\draw  [fill={rgb, 255:red, 0; green, 0; blue, 0 }  ,fill opacity=1 ] (371.25,153.47) .. controls (371.25,151.95) and (372.48,150.72) .. (374,150.72) .. controls (375.52,150.72) and (376.75,151.95) .. (376.75,153.47) .. controls (376.75,154.99) and (375.52,156.22) .. (374,156.22) .. controls (372.48,156.22) and (371.25,154.99) .. (371.25,153.47) -- cycle ;
%Shape: Circle [id:dp1975335906463691] 
\draw  [fill={rgb, 255:red, 0; green, 0; blue, 0 }  ,fill opacity=1 ] (343.92,153.11) .. controls (344,152.08) and (344.91,151.31) .. (345.94,151.39) .. controls (346.97,151.47) and (347.74,152.37) .. (347.66,153.41) .. controls (347.58,154.44) and (346.68,155.21) .. (345.64,155.13) .. controls (344.61,155.05) and (343.84,154.14) .. (343.92,153.11) -- cycle ;
%Shape: Boxed Bezier Curve [id:dp90924505359553] 
\draw [line width=1.5]    (329.67,109.8) .. controls (345.33,114.13) and (373.33,127.13) .. (374,153.47) ;
%Shape: Boxed Bezier Curve [id:dp11533168830826857] 
\draw [line width=1.5]    (329.67,197.13) .. controls (345.33,192.8) and (373.33,179.8) .. (374,153.47) ;
%Curve Lines [id:da45101425151027463] 
\draw [line width=1.5]    (407,155.02) .. controls (410.8,126.64) and (418.52,109.33) .. (437.86,88) ;
%Curve Lines [id:da42988748946226396] 
\draw [line width=1.5]    (407,155.02) .. controls (396.8,96.98) and (370,84.02) .. (345,79.69) ;

% Text Node
\draw (157.71,139.03) node [anchor=south] [inner sep=0.75pt]  [font=\scriptsize]  {$r$};
% Text Node
\draw (48,145.4) node [anchor=north west][inner sep=0.75pt]  [font=\small]  {$\RP^1$};
% Text Node
\draw (349.04,145.37) node [anchor=south] [inner sep=0.75pt]  [font=\scriptsize]  {$r$};
% Text Node
\draw (421.04,146.03) node [anchor=south] [inner sep=0.75pt]  [font=\scriptsize]  {$r$};

\end{tikzpicture}
    \caption{A transformation of type \textbf{(I)} where $\alpha$ is labeled by $r$ and $m=3$.}
    \label{Deform_real_r}
\end{figure}
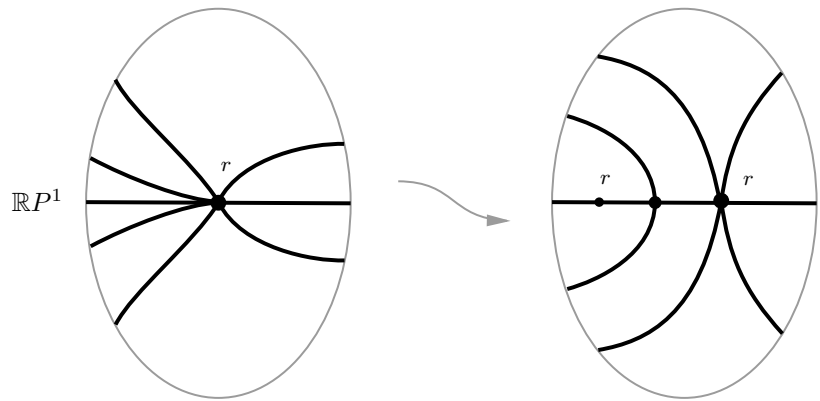

\subsubsection*{Transformation of type (II)}%\label{Sub.Type(II)} 
Consider any vertex $\alpha\in\Gamma\setminus\mathbb{R}P^1$, together with its conjugate $\overline{\alpha}$. 
Assume that \(\alpha\) (and so \(\bar{\alpha}\)) are neighbors of a real non-special node $\sigma$ not labeled by \(r\) and of valency $4$. Consider a small neighborhood $\mathcal{U}_\sigma$ of $\sigma$ such that $\mathcal{U}_\sigma$ contains both $\alpha$ and $\bar{\alpha}$ but no other vertices. We transform $\Gamma$ into a graph $\Gamma'$ as in Figure~\ref{Compl-to-real_cropped}. In $\mathcal{U}_\sigma$, the new graph $\Gamma'$ contains only one vertex $\beta$, which retains the label of $\alpha$. Moreover, the valency of $\beta$ is equal to two times that of $\alpha$.

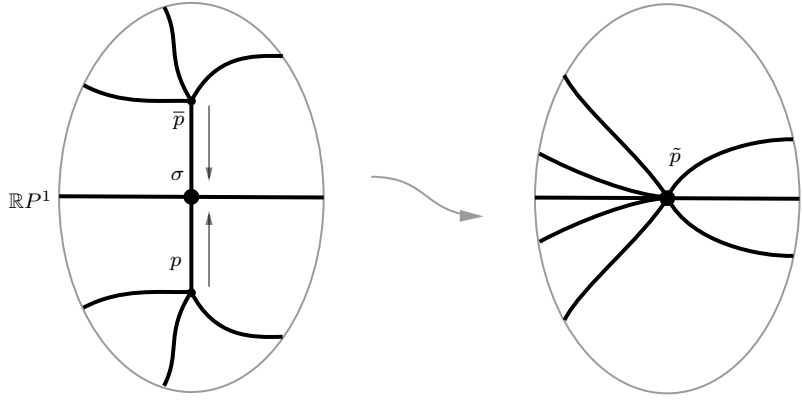
\begin{figure}
    \centering
    \tikzset{every picture/.style={line width=0.75pt}} %set default line width to 0.75pt        

\begin{tikzpicture}[x=0.75pt,y=0.75pt,yscale=-1,xscale=1]
%uncomment if require: \path (0,292); %set diagram left start at 0, and has height of 292

%Shape: Ellipse [id:dp31061451664324147] 
\draw  [color={rgb, 255:red, 155; green, 155; blue, 155 }  ,draw opacity=1 ] (107,153.21) .. controls (107,98.97) and (136.87,55) .. (173.72,55) .. controls (210.57,55) and (240.44,98.97) .. (240.44,153.21) .. controls (240.44,207.44) and (210.57,251.41) .. (173.72,251.41) .. controls (136.87,251.41) and (107,207.44) .. (107,153.21) -- cycle ;
%Straight Lines [id:da5840020261788041] 
\draw [line width=1.5]    (107,152.6) -- (240.44,153.21) ;
%Shape: Circle [id:dp5117891361119118] 
\draw  [fill={rgb, 255:red, 0; green, 0; blue, 0 }  ,fill opacity=1 ] (170.22,152.9) .. controls (170.22,150.97) and (171.79,149.4) .. (173.72,149.4) .. controls (175.66,149.4) and (177.22,150.97) .. (177.22,152.9) .. controls (177.22,154.84) and (175.66,156.4) .. (173.72,156.4) .. controls (171.79,156.4) and (170.22,154.84) .. (170.22,152.9) -- cycle ;
%Straight Lines [id:da9016420832949857] 
\draw [line width=1.5]    (173.72,152.9) -- (173.72,104.5) ;
%Straight Lines [id:da4852590808474594] 
\draw [line width=1.5]    (173.72,201.3) -- (173.72,152.9) ;
%Straight Lines [id:da3750899661096211] 
\draw [color={rgb, 255:red, 0; green, 0; blue, 0 }  ,draw opacity=0.5 ][fill={rgb, 255:red, 128; green, 128; blue, 128 }  ,fill opacity=0.7 ][line width=0.75]    (182.63,143.27) -- (182.63,139.41) -- (182.63,106.87) ;
\draw [shift={(182.63,145.27)}, rotate = 270] [fill={rgb, 255:red, 0; green, 0; blue, 0 }  ,fill opacity=0.5 ][line width=0.08]  [draw opacity=0] (7.2,-1.8) -- (0,0) -- (7.2,1.8) -- cycle    ;
%Straight Lines [id:da04649970823869154] 
\draw [color={rgb, 255:red, 0; green, 0; blue, 0 }  ,draw opacity=0.5 ][fill={rgb, 255:red, 128; green, 128; blue, 128 }  ,fill opacity=0.7 ][line width=0.75]    (182.63,198.27) -- (182.63,161.87) ;
\draw [shift={(182.63,159.87)}, rotate = 90] [fill={rgb, 255:red, 0; green, 0; blue, 0 }  ,fill opacity=0.5 ][line width=0.08]  [draw opacity=0] (7.2,-1.8) -- (0,0) -- (7.2,1.8) -- cycle    ;
%Shape: Circle [id:dp0037611708979174985] 
\draw  [fill={rgb, 255:red, 0; green, 0; blue, 0 }  ,fill opacity=1 ] (171.97,104.5) .. controls (171.97,103.54) and (172.76,102.75) .. (173.72,102.75) .. controls (174.69,102.75) and (175.47,103.54) .. (175.47,104.5) .. controls (175.47,105.47) and (174.69,106.25) .. (173.72,106.25) .. controls (172.76,106.25) and (171.97,105.47) .. (171.97,104.5) -- cycle ;
%Curve Lines [id:da08650556098959572] 
\draw [line width=1.5]    (173.72,104.5) .. controls (185.63,81.41) and (203.63,81.41) .. (219.85,81.86) ;
%Curve Lines [id:da09604311109603603] 
\draw [line width=1.5]    (119.41,96.52) .. controls (138.3,106.67) and (157.5,104.06) .. (173.72,104.5) ;
%Curve Lines [id:da7047726750190031] 
\draw [line width=1.5]    (173.72,104.5) .. controls (160.3,83.34) and (168.74,74) .. (160.07,57.11) ;

%Curve Lines [id:da8865953008916828] 
\draw [line width=1.5]    (173.58,200.88) .. controls (185.49,223.97) and (203.49,223.97) .. (219.71,223.52) ;
%Curve Lines [id:da76266767634939] 
\draw [line width=1.5]    (119.26,208.86) .. controls (138.15,198.71) and (157.36,201.32) .. (173.58,200.88) ;
%Curve Lines [id:da44030665053318097] 
\draw [line width=1.5]    (173.58,200.88) .. controls (160.15,222.04) and (168.6,231.38) .. (159.93,248.26) ;

%Shape: Circle [id:dp9422518282036497] 
\draw  [fill={rgb, 255:red, 0; green, 0; blue, 0 }  ,fill opacity=1 ] (171.97,201.3) .. controls (171.97,200.34) and (172.76,199.55) .. (173.72,199.55) .. controls (174.69,199.55) and (175.47,200.34) .. (175.47,201.3) .. controls (175.47,202.27) and (174.69,203.05) .. (173.72,203.05) .. controls (172.76,203.05) and (171.97,202.27) .. (171.97,201.3) -- cycle ;
%Shape: Ellipse [id:dp8239672890175849] 
\draw  [color={rgb, 255:red, 155; green, 155; blue, 155 }  ,draw opacity=1 ] (347,153.86) .. controls (347,99.62) and (376.87,55.66) .. (413.72,55.66) .. controls (450.57,55.66) and (480.44,99.62) .. (480.44,153.86) .. controls (480.44,208.1) and (450.57,252.07) .. (413.72,252.07) .. controls (376.87,252.07) and (347,208.1) .. (347,153.86) -- cycle ;
%Straight Lines [id:da13801103448943386] 
\draw [line width=1.5]    (347,153.26) -- (480.44,153.86) ;
%Shape: Circle [id:dp8935822086434869] 
\draw  [fill={rgb, 255:red, 0; green, 0; blue, 0 }  ,fill opacity=1 ] (410.22,153.56) .. controls (410.22,151.63) and (411.79,150.06) .. (413.72,150.06) .. controls (415.66,150.06) and (417.22,151.63) .. (417.22,153.56) .. controls (417.22,155.49) and (415.66,157.06) .. (413.72,157.06) .. controls (411.79,157.06) and (410.22,155.49) .. (410.22,153.56) -- cycle ;
%Curve Lines [id:da012170436951348895] 
\draw [line width=1.5]    (413.87,153.53) .. controls (425.77,176.62) and (461.11,183.13) .. (477.33,182.68) ;
%Curve Lines [id:da4390275966216214] 
\draw [line width=1.5]    (349.33,175.54) .. controls (368.22,165.39) and (397.64,153.98) .. (413.87,153.53) ;
%Curve Lines [id:da38935943417019425] 
\draw [line width=1.5]    (413.87,153.53) .. controls (400.44,174.7) and (370.57,198.17) .. (361.9,215.06) ;

%Curve Lines [id:da6556724060963552] 
\draw [color={rgb, 255:red, 155; green, 155; blue, 155 }  ,draw opacity=1 ]   (264.44,142.74) .. controls (294.17,143) and (289.96,161.32) .. (319.27,162.68) ;
\draw [shift={(321.11,162.74)}, rotate = 181.22] [fill={rgb, 255:red, 155; green, 155; blue, 155 }  ,fill opacity=1 ][line width=0.08]  [draw opacity=0] (12,-3) -- (0,0) -- (12,3) -- cycle    ;
%Curve Lines [id:da08921664867009649] 
\draw [line width=1.5]    (413.72,153.01) .. controls (425.63,129.92) and (460.97,123.42) .. (477.19,123.87) ;
%Curve Lines [id:da08114072196003141] 
\draw [line width=1.5]    (349.19,131.01) .. controls (368.08,141.16) and (397.5,152.57) .. (413.72,153.01) ;
%Curve Lines [id:da5497719132372393] 
\draw [line width=1.5]    (413.72,153.01) .. controls (400.3,131.85) and (370.43,108.37) .. (361.76,91.48) ;

% Text Node
\draw (80.13,148.1) node [anchor=north west][inner sep=0.75pt]  [font=\scriptsize]  {$\RP^1$};
% Text Node
\draw (171.72,107.9) node [anchor=north east] [inner sep=0.75pt]  [font=\scriptsize]  {$\overline{p}$};
% Text Node
\draw (170.22,191.4) node [anchor=south east] [inner sep=0.75pt]  [font=\scriptsize]  {$p$};
% Text Node
\draw (171.72,146) node [anchor=south east] [inner sep=0.75pt]  [font=\scriptsize]  {$\sigma $};
% Text Node
\draw (417.71,139.03) node [anchor=south] [inner sep=0.75pt]  [font=\scriptsize]  {$\tilde{p}$};

\end{tikzpicture}
    \caption{A transformation of type \textbf{(II)} where $\alpha$ is a letter $p$, and $m=2$ .}
    \label{Compl-to-real_cropped}
\end{figure}

\subsubsection*{Reduction to simple dessins} As the subsequent section demonstrates, proving Theorem~\ref{Main.Th} ultimately boils down to bounding the number of $r$-labeled vertices of a dessin inside the interval $]0,~1[$. The following lemma ensures that we do not lose generality by considering only simple dessins.

\begin{lemma}\label{lem:to_simple}
Let $\Gamma\subset\CP^1$ be the dessin given by a rational function of the form \(x^a(1-x)^bP^m/Q^m\) with \(a,b,m\) odd integers. Then, there exists a dessin of $\tilde{\Gamma}\subset\CP^1$ 
that satisfies the following three properties:
\begin{enumerate}
    \item the dessin $\tilde{\Gamma}$ is simple with respect to the interval \(]0,1[\),

    \item the sum of the valencies of vertices labeled by $r$ in $\tilde{\Gamma}\ \cap\ ]0,1[$ is no less than that of $\Gamma\ \cap\ ]0,1[$, 

    \item the number of vertices labeled by $p$ (resp.\ by $q$) in $\tilde{\Gamma}$ is equal to the corresponding number in $\Gamma$ (where this number is counted with multiplicity: any special vertex in \(\CP^1 \setminus \{0,1,\infty\}\) of valency \(2km\) counts for \(k\)).
\end{enumerate}
\end{lemma}
\begin{proof}
    We obtain $\tilde{\Gamma}$ from $\Gamma$ by applying the two transformations \textbf{(I)} and \textbf{(II)} in the following way. 
    By making transformations of type \textbf{(I)}, we can ensure that any non-special node has valency exactly four and is not labeled by \(r\), and that all special vertices outside \(\{0,1,\infty\}\) have the same valency \(2m\). This process ensures conditions \textbf{(i)} and \textbf{(iii)}. If there is no nodes $\alpha\in\Gamma\setminus\mathbb{R}P^1$ neighbor of a non-special real node, condition \textbf{(ii)} is also satisfied, and we are done for the first point. Otherwise, we perform the transformation of type \textbf{(II)}, this creates one node in $\mathbb{R}P^1$ which can violate one of conditions \textbf{(i)} or \textbf{(iii)}. In this case, we perform a transformation of type \textbf{(I)} around this real node to recover \textbf{(i)} and \textbf{(iii)}. Repeating this process sufficiently many times gives us eventually conditions \textbf{(i)}, \textbf{(ii)} and \textbf{(iii)}. Indeed, this is so since these transformations do not give rise to vertices $\alpha\in\Gamma\setminus\mathbb{R}P^1$, and there are finitely many of them. As such points will be exhausted after sufficiently many transformations of type \textbf{(II)}, the overall sequence of transformations will ultimately finish and the resulting dessin is simple with respect to the interval \(]0,1[\).

    Finally, the transformations \textbf{(I)} and \textbf{(II)} maintain the last two points of the lemma.
\end{proof}

Note that a special vertex of \(\Gamma\) either belongs to \(\{0,1,\infty\}\), is a root of \(P\), or a root of \(Q\). Hence the number of distinct special vertices of \(\tilde{\Gamma}\) outside \(\{0,1,\infty\}\) is bounded by \(\deg(P)+\deg(Q)\).

\section{Proof of Theorem~\ref{Main.Th}}\label{sec:proof_rational}
We start by showing the following useful result. For any $r,s\in\N$, every point\\ $\bm{c}:=(a_0,\ldots,a_r,b_0,\ldots,b_s,\alpha,\beta)\in \R^{r+s+4}$ with \(\bm{a}:=(a_0,\ldots,a_r) \ne \underline{0}\) and \(\bm{b}:=(b_0,\ldots,b_s)\ne \underline{0}\) determines an analytic function on \(]0,1[\) of the form 
\begin{equation}\label{eq:phunc}
    \phunc(x)= x^\alpha~(1-x)^\beta~P/Q,
\end{equation} where $P,Q\in\R[x]$ are expressed as $P:=a_0 + a_1x~+\cdots~+~a_rx^r$, $Q:=b_0 ~+~ b_1x~+\cdots~+~b_sx^s$. In what follows, we use \(M:= \{(\bm{a},\bm{b},\alpha,\beta) \in \R^{r+s+4} \mid \bm{a}, \bm{b} \ne \underline{0}\}\) to denote the space of all such functions~\eqref{eq:phunc}.

\begin{proposition}\label{prp:analytic_to_rational}
Let $I:=\ ]0,1[$. Then, for every point ${\bm{p}}\in M$, 
there exists ${\bm{q}}:=(\bm{a},\bm{b},\alpha,\beta)\in \Q^{r+s+4} \cap M$, such that the below inequality holds
\begin{align*}
\#\left\lbrace x\in I~|~ \phi_{\bm{p}} - 1 = 0\right\rbrace & ~\leq \#\left\lbrace x\in I~|~ \phi_{\bm{q}} - 1 = 0\right\rbrace,
\end{align*} and for each fraction $\frac{i}{j}\in\{\alpha,\beta\}$, the numerator $i$ and the denominator $j$ are odd integers.
\end{proposition}

\begin{proof}
As the result becomes trivial otherwise, we assume that \(\phi_{\bm{c}}\) is not constant. 
Note that $\phi_{\bm{c}}-1$ has the same number of roots in $]0,1[$ as $x^\alpha(1-x)^\beta P - Q$. 
By making a sequence of $\delta = \deg Q +1$ derivations on the latter expression, we obtain $x^{\alpha-\delta}(1-x)^{\beta-\delta}~H$ for some polynomial $H\in\R[x]$. Since $H$ has finitely-many solutions in $]0,1[$, by Rolle Theorem, so does \(\phi_{\bm{c}}-1\).

We can assume without loss of generality that the limits of \(\phi_{\bm{c}}\) in \(0\) and \(1\) are not \(1\). Indeed, if \(\alpha \beta \ne 0\), it is enough to assume (by dividing \(P\) or \(Q\)) that for ${\bm{p}}\in M$, neither $0$, nor $1$ is a root of $P\cdot Q$. Otherwise, if \(\alpha=0\), then, it is sufficient to replace \(x\) in~\eqref{eq:phunc} by \((1-\varepsilon)x+\varepsilon\) for \(\varepsilon>0\) sufficiently small, and if \(\beta =0\), we can replace \(x\) by \((1-\varepsilon)x\).  

Let $X$ be a submanifold of $M\times I$ given by 
\begin{align*}
X & ~:=\left\lbrace (\bm{c},~x)~|~\Phi(\bm{c},~x) - 1 = 0\right\rbrace,
\end{align*} where $\Phi$ is the function $ M\times\ ]0,1[\ \longrightarrow\R$, $(\bm{c},~x) \longmapsto \phunc(x)$. Clearly, the set $X$ is smooth since the partial derivative $\partial \Phi/\partial a_0 = x^\alpha (1-x)^\beta/Q$ is a non-vanishing function on its domain $M\times I$.

Thanks to Rolle's Theorem, the restricted projection $
\left. \pi\right|_X: X \longrightarrow M$, $(\bm{c},~x) \longmapsto \bm{c}$ is a finite-to-one map with fibers having at most $r+2s$ points. Furthermore, if $\cR\subset M$ is the subset, containing the parameter tuple ${\bm{p}}$, and given by 
\begin{align*}
\cR& ~:=\left\lbrace\bm{c}~|~(P\cdot Q)(0)\neq 0 \text{ and } (P\cdot Q)(1)\neq 0\right\rbrace, 
\end{align*} then the map 
\begin{align*}
\Proj:=\left. \pi\right|_{X\cap\pi^{-1}(\cR)}: & ~X\cap\pi^{-1}(\cR) \longrightarrow \cR
\end{align*} is proper (since we assumed that the limits of \(\phi_{\bm{c}}\) in \(0\) and \(1\) are not \(1\)). Therefore, this is a branched covering of $\cR$ with some branching locus $B\subset\cR$. Consequently, to maximize the size of the fibers under $\Proj$, we  do not lose generality by assuming that ${\bm{p}}\in\cR\setminus B$. Then, all fibers in the vicinity of ${\bm{p}}$ have the same size. This finishes the proof as any neighborhood of ${\bm{p}}$ can be approximated by a point ${\bm{q}}$ with rational coordinates. Furthermore, one can choose ${\bm{q}}$ so that the odd conditions on the numerators and denominators of $\alpha$ and $\beta$ are satisfied.
\end{proof}

In what follows, we fix a function $\phi:=\phunc$ as defined in the beginning of this section for an arbitrary choice of $\bm{c}\in M$. We show that the number of roots of $\phi - 1$ in $]0,1[$ cannot exceed $\deg P + \deg Q +2$. Thanks to  Proposition~\ref{prp:analytic_to_rational}, 
we may assume that $\alpha,\beta\in\Q$ and for every $i/j\in\{\alpha,\beta\}$, each of $i$ and $j$ are odd. Clearing out the denominators of $\alpha$ and $\beta$ by raising $\phi$ to the power of some odd $m\in\N$, we obtain a rational function $\phi^m=x^a(1-x)^bP^m/Q^m\in\R(x)$, where each of $a$, $b$, and $m$ are odd. As the number of solutions in $]0,1[$ does not decrease, we prove the result for $\varphi:=\phi^m$ instead of $\phi$.
Let $\Gamma\subset\C P^1$ denote the dessin of $\varphi$. The number of vertices labeled by \(r\) in \(]0,1[\) is at least the number of roots of \(\phi-1\) there, and the number of vertices \(v\in \CP^1\setminus\{0,1,\infty\}\) labeled by \(p\) (resp. by \(q\)) equals the degree of \(P\) (resp. \(Q\)). That is, if a vertex $v$ above has valency $2km$, then it corresponds to a root of $P$ or of $Q$ (depending on its label) having multiplicity $k$.

The proof of Theorem~\ref{Main.Th} culminates in Proposition~\ref{prp:interior_exterior} below. 
Note that, if $\Gamma$ is simple w.r.t. $J$, all special vertices of \(\Gamma \setminus (\RP^1 \setminus J)\) have the same valency $2m$.  

This observation gives rise to the following important lemma for proving Proposition~\ref{prp:interior_exterior}. 

\begin{lemma}\label{L.no.p_0}
    Let \(\Gamma\) simple with respect to some interval \(J\).
    Let $\alpha, \beta \in J$ be nodes of $\Gamma$ which are neighbors. Then, we have that \(\alpha\) is special if and only if \(\beta\) is special.
\end{lemma}

\begin{proof}
We argue by contradiction. Assume that there exist $\alpha$ non-special and $\beta$ special, neighbors in \(J\). Consider such a couple \((\alpha,\beta)\), whose corresponding interval \(]\alpha,\beta[\) is the smallest possible with this property. If \(\gamma\) is a non-special node of \(]\alpha,\beta[\), then by condition \textbf{(ii)}, \(\gamma\) has a real neighbor \(\delta\) in \(]\alpha,\beta[\) which is non-special (by minimality of the couple \((\alpha,\beta)\)). 
Let \(\tilde{\mathcal{D}}\) be the closed disk containing the segment \([\gamma,\delta]\) whose border is given by the arc in \(\Gamma \setminus \RP^1\) joining \(\gamma\) to \(\delta\) and by its conjugate arc. Contracting \(\tilde{\mathcal{D}}\) to a single point does not affect (up to homeomorphism) the closure of the connected components of \(\CP^1 \setminus \Gamma\) outside \(\tilde{\mathcal{D}}\). So we can assume that \(]\alpha,\beta[\) contains only special nodes.

Consider the open disk $\mathcal{D}$ which contains $]\alpha,\beta[$ and which is bounded by the complex branch $\gamma$ of $\Gamma$ joining $\alpha$ to $\beta$ together with the conjugate branch, $\overline{\gamma}$. Since $\alpha$ and $\beta$ are neighbors, the corresponding two conjugate branches can be chosen such that they do not contain any vertices except \(\alpha\) and \(\beta\). Consider the set \(V\) of nodes in $\mathcal{D}\cup\{\beta\}$ (which are all special by hypothesis) together with the branches of $\Gamma\cap(\mathcal{D}\cup\{\beta\})$ joining them.

This gives a graph $\mathfrak{G}$ of vertices \(V\). 
Since any node in \(V\setminus \{\beta\}\) of $\Gamma$ has its neighbors in $\Gamma\cap(\mathcal{D}\cup\{\beta\}) \subset \Gamma \setminus (\RP^1\setminus J)$, its valency in $\mathfrak{G}$ is exactly \(2m\). 
Assume that \(\beta\) is labeled by \(p\) (the case where \(\beta\) is labeled by \(q\) is symmetric).
By the cycle rule of the dessins, the graph is bipartite: one part \(V_P\) corresponds to vertices labeled by $p$ and the other part, \(V_Q\), consists of vertices labeled by $q$. Then, the sum of the valencies of the vertices of \(V_Q\) (which is a multiple of \(2m\)) has to be equal to the sum of the degrees of the vertices of \(V_P\) (which is \(d_{\beta}\) plus a multiple of \(2m\)). However, it contradicts the fact that the valency \(\val(p)\) of \(\beta\) in $\mathfrak{G}$ is in \([1,2m-1]\). 
\end{proof}

The proof of Theorem~\ref{Main.Th} concludes at the end of this section, and is based on the following proposition.

\begin{proposition}\label{prp:interior_exterior} Let \(J\) be a strict open interval of \(\CP^1\).
Let $\Gamma$ be a simple dessin with respect to \(J\) as in Definition~\ref{def:simple_pair}, and let \(\cS\) be the set of the special vertices. Assume that the valency of each vertex \(x \in J\cap \cS\) is not a multiple of \(4\). 
If $\cRJ$ denotes the set of vertices of $\Gamma\cap J$ with label $r$, then the following  relation holds
\begin{equation}\label{eq:RJleqSm}
\#\cRJ \leq \#\cS - 1.
\end{equation} 
\end{proposition} 

\begin{proof}
Let $\cCJ$ denote the set of non-special nodes of \(J\). 
Since $\Gamma$ is simple, any $c\in\cCJ$ is a node in $J \setminus \cS$ of valency exactly four. By~\textbf{(i)}, we have \(\cCJ \cap \cRJ = \emptyset\).

We argue by induction on the number $k$ of points in $\cCJ$. Assume first that $k=0$. 
Thanks to the cycle rule, the dessin contains at least one vertex labeled by \(p\), and one by \(q\), so \(\#\cS \ge 2\). Thus the result holds if \(\#\cRJ \le 1\).
Between two consecutive elements \(r_1\) and \(r_2\) of \(\cRJ\), any real node is in \(\cS\) and thus has valency \(2m\) for some odd integer \(m\). Then, if $x_1,\ldots,x_s\subset J$ is the collection of those nodes, up to reversing the orientation, the segment \([r_1,r_2]\) has orientation of the form 
\[
r_1\rightarrow x_1 \rightarrow \cdots\rightarrow x_s\rightarrow r_2.
\]  Hence, the cycle rule $r\rightarrow p\rightarrow q\rightarrow r$ continues to hold for the segment \([r_1,r_2]\). 
% In particular, 
Therefore, this segment should contain at least one vertex labelled by \(p\), and one by \(q\). Consequently, the result is immediate when \(\#\cRJ \ge 3\). 

Assume now that \(\cRJ = \{r_1,r_2\}\) (i.e., $\#\cRJ=2$). In what follows, we suppose that $\#\cS = 2$, and show that this leads to a contradiction. From assumptions above, there are exactly two special vertices in \(\cS\) which are, by the previous argument, between \(r_1\) and \(r_2\): 
\[
    r_1\rightarrow x_p \rightarrow x_q\rightarrow r_2
\] (the case where \(x_p\) and \(x_q\) are swapped is symmetric). 
Let $\sigma$ be the closest node in $\RP^1$ to the right hand side of $r_2$. By the cycle rule, the node \(\sigma\) should belong to the complementary interval \(]r_2,r_1[\  \subseteq \RP^1\) (since otherwise \(r_1\) and \(r_2\) would be adjacent). In particular, $\sigma$ is non-special, and from~\textbf{(iii)}, it has a real neighbor \(y\). The (lower) arc \(\gamma\) from \(\sigma\) to \(y\) splits the lower hemisphere $H\subset\CP^1\setminus \RP^1$ into two components \(H^1\) and $H^2$ (i.e., we have $H^1\cup\gamma\cup H^2 = H$). Then, the vertex \(y\) is the only point in \(\overline{H^1} \cap \overline{H^2}\) which can be special. In particular either \(x_p\) or \(x_q\) is not \(y\) and does not belong in some \(\overline{H^i}\). By the cycle rule, the component \(\overline{H^i}\) should contain a third special vertex. This contradicts $\#\cS=2$.

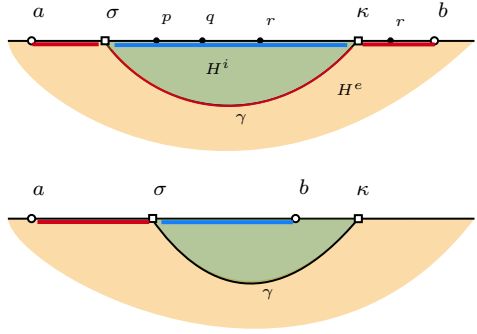
\begin{figure}
\centering
\tikzset{every picture/.style={line width=0.75pt}} %set default line width to 0.75pt        

\begin{tikzpicture}[x=0.75pt,y=0.75pt,yscale=-1,xscale=1]
%uncomment if require: \path (0,292); %set diagram left start at 0, and has height of 292

%Curve Lines [id:da22632176848759467] 
\draw [fill={rgb, 255:red, 65; green, 117; blue, 5 }  ,fill opacity=0.4 ]   (143.9,40.6) .. controls (169.61,77.38) and (229.63,90.68) .. (271.63,40.68) ;
%Curve Lines [id:da8432819761458847] 
\draw [color={rgb, 255:red, 245; green, 166; blue, 35 }  ,draw opacity=0 ][fill={rgb, 255:red, 245; green, 166; blue, 35 }  ,fill opacity=0.4 ]   (125.02,40.6) .. controls (121,40.55) and (102,40.55) .. (95.44,40.68) .. controls (98.78,67.8) and (214.33,150.91) .. (330.77,40.68) .. controls (322.22,40.58) and (300.56,41.24) .. (272.19,40.68) .. controls (242.42,75.51) and (188.11,91.8) .. (143.9,40.6) ;
%Straight Lines [id:da8286025397902969] 
\draw    (95.44,40.68) -- (330.77,40.68) ;
%Shape: Circle [id:dp7903447659776817] 
\draw  [fill={rgb, 255:red, 0; green, 0; blue, 0 }  ,fill opacity=1 ] (169.24,40.5) .. controls (169.24,39.88) and (169.74,39.37) .. (170.37,39.37) .. controls (170.99,39.37) and (171.5,39.88) .. (171.5,40.5) .. controls (171.5,41.12) and (170.99,41.63) .. (170.37,41.63) .. controls (169.74,41.63) and (169.24,41.12) .. (169.24,40.5) -- cycle ;
%Shape: Circle [id:dp713812967268163] 
\draw  [fill={rgb, 255:red, 0; green, 0; blue, 0 }  ,fill opacity=1 ] (192.4,40.5) .. controls (192.4,39.88) and (192.91,39.37) .. (193.54,39.37) .. controls (194.16,39.37) and (194.67,39.88) .. (194.67,40.5) .. controls (194.67,41.12) and (194.16,41.63) .. (193.54,41.63) .. controls (192.91,41.63) and (192.4,41.12) .. (192.4,40.5) -- cycle ;
%Shape: Circle [id:dp1752025569088992] 
\draw  [fill={rgb, 255:red, 0; green, 0; blue, 0 }  ,fill opacity=1 ] (221.57,40.5) .. controls (221.57,39.88) and (222.08,39.37) .. (222.7,39.37) .. controls (223.33,39.37) and (223.83,39.88) .. (223.83,40.5) .. controls (223.83,41.12) and (223.33,41.63) .. (222.7,41.63) .. controls (222.08,41.63) and (221.57,41.12) .. (221.57,40.5) -- cycle ;
%Shape: Circle [id:dp7403312641210422] 
\draw  [fill={rgb, 255:red, 0; green, 0; blue, 0 }  ,fill opacity=1 ] (287.29,40.5) .. controls (287.29,39.88) and (287.79,39.37) .. (288.42,39.37) .. controls (289.04,39.37) and (289.55,39.88) .. (289.55,40.5) .. controls (289.55,41.12) and (289.04,41.63) .. (288.42,41.63) .. controls (287.79,41.63) and (287.29,41.12) .. (287.29,40.5) -- cycle ;
%Shape: Circle [id:dp7964977144166155] 
\draw  [color={rgb, 255:red, 0; green, 0; blue, 0 }  ,draw opacity=1 ][fill={rgb, 255:red, 255; green, 255; blue, 255 }  ,fill opacity=1 ] (308.71,40.63) .. controls (308.71,39.62) and (309.54,38.8) .. (310.55,38.8) .. controls (311.56,38.8) and (312.38,39.62) .. (312.38,40.63) .. controls (312.38,41.64) and (311.56,42.46) .. (310.55,42.46) .. controls (309.54,42.46) and (308.71,41.64) .. (308.71,40.63) -- cycle ;
%Shape: Circle [id:dp8791661400918988] 
\draw  [color={rgb, 255:red, 0; green, 0; blue, 0 }  ,draw opacity=1 ][fill={rgb, 255:red, 255; green, 255; blue, 255 }  ,fill opacity=1 ] (105.57,40.63) .. controls (105.57,39.62) and (106.39,38.8) .. (107.4,38.8) .. controls (108.42,38.8) and (109.24,39.62) .. (109.24,40.63) .. controls (109.24,41.64) and (108.42,42.46) .. (107.4,42.46) .. controls (106.39,42.46) and (105.57,41.64) .. (105.57,40.63) -- cycle ;
%Curve Lines [id:da8941188807631043] 
\draw [color={rgb, 255:red, 245; green, 166; blue, 35 }  ,draw opacity=0 ][fill={rgb, 255:red, 245; green, 166; blue, 35 }  ,fill opacity=0.4 ]   (125.02,130.31) .. controls (121,130.26) and (102,130.26) .. (95.44,130.4) .. controls (98.78,157.51) and (247.2,240.57) .. (330.77,130.4) .. controls (322.22,130.29) and (300.56,130.96) .. (272.19,130.4) .. controls (243.2,163.63) and (204.8,180.17) .. (168,130.43) ;
%Straight Lines [id:da534876263297857] 
\draw    (95.44,130.4) -- (330.77,130.4) ;
%Curve Lines [id:da14644397110754026] 
\draw [fill={rgb, 255:red, 65; green, 117; blue, 5 }  ,fill opacity=0.4 ]   (168,130.43) .. controls (193.7,167.22) and (230.19,180.4) .. (272.19,130.4) ;
%Shape: Circle [id:dp20682538427539965] 
\draw  [color={rgb, 255:red, 0; green, 0; blue, 0 }  ,draw opacity=1 ][fill={rgb, 255:red, 255; green, 255; blue, 255 }  ,fill opacity=1 ] (238.71,130.35) .. controls (238.71,129.33) and (239.54,128.51) .. (240.55,128.51) .. controls (241.56,128.51) and (242.38,129.33) .. (242.38,130.35) .. controls (242.38,131.36) and (241.56,132.18) .. (240.55,132.18) .. controls (239.54,132.18) and (238.71,131.36) .. (238.71,130.35) -- cycle ;
%Shape: Circle [id:dp3292296236097595] 
\draw  [color={rgb, 255:red, 0; green, 0; blue, 0 }  ,draw opacity=1 ][fill={rgb, 255:red, 255; green, 255; blue, 255 }  ,fill opacity=1 ] (105.57,130.35) .. controls (105.57,129.33) and (106.39,128.51) .. (107.4,128.51) .. controls (108.42,128.51) and (109.24,129.33) .. (109.24,130.35) .. controls (109.24,131.36) and (108.42,132.18) .. (107.4,132.18) .. controls (106.39,132.18) and (105.57,131.36) .. (105.57,130.35) -- cycle ;
%Shape: Square [id:dp2129144636619852] 
\draw  [fill={rgb, 255:red, 255; green, 255; blue, 255 }  ,fill opacity=1 ] (166.69,128.53) -- (170.24,128.53) -- (170.24,132.09) -- (166.69,132.09) -- cycle ;
%Shape: Square [id:dp6379156363333545] 
\draw  [fill={rgb, 255:red, 255; green, 255; blue, 255 }  ,fill opacity=1 ] (270.41,128.62) -- (273.97,128.62) -- (273.97,132.17) -- (270.41,132.17) -- cycle ;
%Straight Lines [id:da643650915137661] 
\draw [color={rgb, 255:red, 208; green, 2; blue, 27 }  ,draw opacity=1 ][line width=1.5]    (107.91,42.38) -- (141.24,42.38) ;
%Straight Lines [id:da1031313789319982] 
\draw [color={rgb, 255:red, 208; green, 2; blue, 27 }  ,draw opacity=1 ][line width=1.5]    (274.41,42.46) -- (310.75,42.46) ;
%Straight Lines [id:da15390244376257645] 
\draw [color={rgb, 255:red, 22; green, 126; blue, 247 }  ,draw opacity=1 ][line width=1.5]    (149.24,42.78) -- (266.58,42.78) ;
%Straight Lines [id:da6889301222563111] 
\draw [color={rgb, 255:red, 208; green, 2; blue, 27 }  ,draw opacity=1 ][line width=1.5]    (110.35,132.09) -- (166.69,132.09) ;
%Straight Lines [id:da6743958949125558] 
\draw [color={rgb, 255:red, 22; green, 126; blue, 247 }  ,draw opacity=1 ][line width=1.5]    (172.88,132.18) -- (239.21,132.18) ;
%Curve Lines [id:da12344353611969439] 
\draw [color={rgb, 255:red, 208; green, 2; blue, 27 }  ,draw opacity=1 ]   (147.42,45.01) .. controls (175.67,79.01) and (228.42,87.26) .. (268.92,44.01) ;
%Shape: Square [id:dp30644811983805187] 
\draw  [fill={rgb, 255:red, 255; green, 255; blue, 255 }  ,fill opacity=1 ] (142.69,38.82) -- (146.24,38.82) -- (146.24,42.38) -- (142.69,42.38) -- cycle ;
%Shape: Square [id:dp011098026725271715] 
\draw  [fill={rgb, 255:red, 255; green, 255; blue, 255 }  ,fill opacity=1 ] (270.41,38.9) -- (273.97,38.9) -- (273.97,42.46) -- (270.41,42.46) -- cycle ;

% Text Node
\draw (148.36,30.44) node [anchor=south] [inner sep=0.75pt]  [font=\scriptsize]  {$\sigma $};
% Text Node
\draw (274.47,30.14) node [anchor=south] [inner sep=0.75pt]  [font=\scriptsize]  {$\kappa $};
% Text Node
\draw (175.54,33.47) node [anchor=south] [inner sep=0.75pt]  [font=\tiny]  {$p$};
% Text Node
\draw (197.5,33.72) node [anchor=south] [inner sep=0.75pt]  [font=\tiny]  {$q$};
% Text Node
\draw (228,34.31) node [anchor=south] [inner sep=0.75pt]  [font=\tiny]  {$r$};
% Text Node
\draw (201.33,58.98) node [anchor=south] [inner sep=0.75pt]  [font=\tiny]  {$H^{i}$};
% Text Node
\draw (268.15,69.59) node [anchor=south] [inner sep=0.75pt]  [font=\tiny]  {$H^{e}$};
% Text Node
\draw (293.43,34.89) node [anchor=south] [inner sep=0.75pt]  [font=\tiny]  {$r$};
% Text Node
\draw (314.86,29.89) node [anchor=south] [inner sep=0.75pt]  [font=\scriptsize]  {$b$};
% Text Node
\draw (111.05,30.03) node [anchor=south] [inner sep=0.75pt]  [font=\scriptsize]  {$a$};
% Text Node
\draw (213.14,76.31) node [anchor=north] [inner sep=0.75pt]  [font=\tiny]  {$\gamma $};
% Text Node
\draw (172.36,120.15) node [anchor=south] [inner sep=0.75pt]  [font=\scriptsize]  {$\sigma $};
% Text Node
\draw (274.47,119.85) node [anchor=south] [inner sep=0.75pt]  [font=\scriptsize]  {$\kappa $};
% Text Node
\draw (244.86,119.6) node [anchor=south] [inner sep=0.75pt]  [font=\scriptsize]  {$b$};
% Text Node
\draw (111.05,119.74) node [anchor=south] [inner sep=0.75pt]  [font=\scriptsize]  {$a$};
% Text Node
\draw (226.64,172.42) node [anchor=south] [inner sep=0.75pt]  [font=\tiny]  {$\gamma $};

\end{tikzpicture}
\caption{Examples of a splitting of $H$ into the two open discs $H^1$ and $H^2$. 
Depending whether or not $\sigma$ is in $J$, each of the intervals in red and 
blue satisfy the induction hypothesis of the Proof of Proposition~\ref{prp:interior_exterior}.}\label{fig:dessins_splitting}
\end{figure}

Assume that~\eqref{eq:RJleqSm} holds true whenever $\#\cCJ=k-1$. We will show that this is true for $\#\cCJ=k$. Thanks to Condition~\textbf{(ii)} of $\Gamma$ being simple, any point $\sigma\in\cCJ$ is a neighbor to a node $\kappa\in\Gamma\cap\RP^1$.

Similarly as above, the lower arc $\gamma\subset\Gamma$ joining $\sigma$ to $\kappa$ splits $H\subset\CP^1\setminus \RP^1$ into two components $H^1$ and $H^2$ (see Figure~\ref{fig:dessins_splitting}). If \(\sigma\) and \(\kappa\) belong to \(J\), the first component \(H^1\) is the one whose border contains the subsegment of \(J\) delimited by \(\sigma\) and \(\kappa\). If \(\kappa \notin J\), the roles of \(H^1\) and \(H^2\) are symmetric.

In turn, points in $\cRJ$ are split into two disjoint subsets, denoted by $\cRJ^1:=\cRJ\cap \overline{H^1}$, and $\cRJ^2:=\cRJ\cap \overline{H^2}$. Write $J$ as the interval $]a,b[$ for some $a<b$. 
The decomposition splits into two dessins $\Gamma=\Gamma^1\cup\Gamma^2$, where \(\Gamma^1\) lies in the closure of \(H^1\), and \(\Gamma^2\) in the closure of \(H^2\). Let \(\mathcal{S}^i\) be the set of special nodes of \(\Gamma^i\).
We consider two situations; whether $\kappa\in J$, or $\kappa\not\in J$.

Assume first that $\kappa\in J$. Without loss of generality, we assume furthermore that $\sigma<\kappa$. 
By Lemma~\ref{L.no.p_0}, $\kappa\in\cCJ$ as well. In particular, both sets $\cS^1:=\cS\cap \overline{H^1}$, and $\cS^2:=\cS\cap \overline{H^2}$ are disjoint. Moreover the valency of the points of \(\cS\) is unchanged. 

Therefore, applying the induction hypothesis on the dessin \(\Gamma^1\) with \(J^1 =\  ]\sigma,\kappa[\), and on \(\Gamma^2\) with \(J^2 =\ ]a,\sigma]\cup\gamma\cup[\kappa,b[\), we get the following relations
\begin{equation}
\label{eq:main_relations}
\begin{array}{cccclll}
	\#\cRJ^1 &\leq & \#\cS^1-~1\\[4pt]
	\#\cRJ^2 &\leq & \#\cS^2-~1,
\end{array}
 \end{equation} 
 which imply the proposition.

Assume now that $\kappa\not\in J$ (see Figure~\ref{fig:dessins_splitting}). If \(\kappa \in \cCJ\), we apply the same argument with the intervals \(J^1 =\ ]\sigma,b[\) and \(J^2 =\ ]a,\sigma[\). If \(\kappa \in \cS\), we do the same approach but we duplicate \(\kappa\) in \(\cS^1\) and \(\cS^2\) (since \(\kappa \notin J^1 \cup J^2\), the valencies of the points in \(\cS^1 \cap J^1\) and in \(\cS^2 \cap J^2\) are still not multiples of four). In particular, we have \(\#\cS + 1 = \cS^1 + \cS^2\). By induction hypothesis, we still have~\eqref{eq:main_relations}, which terminates the induction.
\end{proof}

By Lemma~\ref{lem:to_simple}, 
we may assume that 
$\Gamma$ is simple with respect to the interval \(]0,1[\) (see Definition~\ref{def:simple_pair}).
To finish the proof of Theorem~\ref{Main.Th}, we proceed as follows. Let $\cS $ be the set $\varphi^{-1}(\{0,\infty\})$, and $J$ be the interval $]0,1[$. 
In particular, \(\mathcal{S} \subseteq \{0,1,\infty\}\cup\{z \mid P(z)=0\}\cup\{z \mid Q(z)=0\}\), and its cardinal is bounded by \(\deg(P)+\deg(Q)+3\).
Since $m$ is odd, the hypothesis of Proposition~\ref{prp:interior_exterior} is satisfied for the graph $\Gamma$. Therefore, we get 
\begin{align}\label{eq:important_inequality}
\#\{x\in\ ]0,1[\ |\ \varphi(x) = 1\} = \#\cRJ \leq\#\cS - 1\leq\deg P +\deg Q + 2.
\end{align}

\bibliographystyle{abbrv}

\begin{thebibliography}{10}

\bibitem{BBS06}
B.~Bertrand, F.~Bihan, and F.~Sottile.
\newblock Polynomial systems with few real zeroes.
\newblock {\em Math. Z.}, 253(2):361--385, 2006.

\bibitem{B07}
F.~Bihan.
\newblock Polynomial systems supported on circuits and dessins d'enfants.
\newblock {\em J. Lond. Math. Soc. (2)}, 75(1):116--132, 2007.

\bibitem{BRS08}
F.~Bihan, J.-M. Rojas, and F.~Sottile.
\newblock On the sharpness of fewnomial bounds and the number of components of
  fewnomial hypersurfaces.
\newblock In {\em Algorithms in algebraic geometry}, pages 15--20. Springer,
  2008.

\bibitem{MR3892410}
F.~Bihan, F.~Santos, and P.-J. Spaenlehauer.
\newblock A polyhedral method for sparse systems with many positive solutions.
\newblock {\em SIAM J. Appl. Algebra Geom.}, 2(4):620--645, 2018.

\bibitem{BS07}
F.~Bihan and F.~Sottile.
\newblock New fewnomial upper bounds from {G}ale dual polynomial systems.
\newblock {\em Mosc. Math. J.}, 7(3):387--407, 573, 2007.

\bibitem{bihan2011fewnomial}
F.~Bihan and F.~Sottile.
\newblock Fewnomial bounds for completely mixed polynomial systems.
\newblock {\em Adv. Geom}, 11:541--556, 2011.

\bibitem{BR90}
O.~Bottema and B.~Roth.
\newblock {\em Theoretical kinematics}.
\newblock Dover Publications, Inc., New York, 1990.
\newblock Corrected reprint of the 1979 edition.

\bibitem{Br06}
E.~Brugall{\'e}.
\newblock Real plane algebraic curves with asymptotically maximal number of
  even ovals.
\newblock {\em Duke Math. J.}, 131(3):575--587, 2006.

\bibitem{By89}
C.~I. Byrnes.
\newblock Pole assignment by output feedback.
\newblock In {\em Three decades of mathematical system theory}, volume 135 of
  {\em Lecture Notes in Control and Inform. Sci.}, pages 31--78. Springer,
  Berlin, 1989.

\bibitem{Des1637}
R.~Descartes.
\newblock {La G{\'e}om{\'e}trie}.
\newblock 1637.

\bibitem{Hil16}
B.~El~Hilany.
\newblock {\em {G{\'e}om{\'e}trie tropicale et syst{\`e}mes polynomiaux, PhD
  Thesis, 2016}}.
\newblock PhD thesis, {Universit{\'e} Grenoble Alpes (ComUE)\\
  https://boulos-elhilany.com}, 2016.

\bibitem{EH18}
B.~El~Hilany.
\newblock Constructing polynomial systems with many positive solutions using
  tropical geometry.
\newblock {\em Rev. Mat. Complut}, 31(2):525--544, 2018.

\bibitem{GH02}
K.~Gatermann and B.~Huber.
\newblock A family of sparse polynomial systems arising in chemical reaction
  systems.
\newblock {\em J. Symbolic Comput.}, 33(3):275--305, 2002.

\bibitem{H02}
B.~Haas.
\newblock A simple counterexample to {K}ouchnirenko's conjecture.
\newblock {\em Beitr\"age Algebra Geom.}, 43(1):1--8, 2002.

\bibitem{IZ18}
I.~Itenberg and D.~Zvonkine.
\newblock Hurwitz numbers for real polynomials.
\newblock {\em Comment. Math. Helv.}, 93(3):441--474, 2018.

\bibitem{Kh91}
A.~G. Khovanski{\u\i}.
\newblock {\em Fewnomials}, volume~88 of {\em Translations of Mathematical
  Monographs}.
\newblock American Mathematical Society, Providence, RI, 1991.
\newblock Translated from the Russian by Smilka Zdravkovska.

\bibitem{KPT15}
P.~Koiran, N.~Portier, and S.~Tavenas.
\newblock A {W}ronskian approach to the real {$\tau$}-conjecture.
\newblock {\em J. Symbolic Computation}, 68~(part 2):195--214, 2015.

\bibitem{Ku08}
A.~Kushnirenko.
\newblock {A letter to Frank Sottile (2008)}.
\newblock {\em {\\
  https://franksottile.github.io//research/pdf/Kushnirenko.pdf}}.

\bibitem{LRW03}
T.-Y. Li, J.~M. Rojas, and X.~Wang.
\newblock Counting real connected components of trinomial curve intersections
  and {$m$}-nomial hypersurfaces.
\newblock {\em Discrete Comput. Geom.}, 30(3):379--414, 2003.

\bibitem{MFRGCSD16}
S.~M{\"u}ller, E.~Feliu, G.~Regensburger, C.~Conradi, A.~Shiu, and
  A.~Dickenstein.
\newblock Sign {C}onditions for {I}njectivity of {G}eneralized {P}olynomial
  {M}aps with {A}pplications to {C}hemical {R}eaction {N}etworks and {R}eal
  {A}lgebraic {G}eometry.
\newblock {\em Found. Comput. Math.}, 16(1):69--97, 2016.

\bibitem{muller2023parametrized}
S.~M{\"u}ller and G.~Regensburger.
\newblock Parametrized systems of polynomial equations with real exponents:
  applications to fewnomials.
\newblock {\em arXiv preprint arXiv:2304.05273}, 2023.

\bibitem{NY95}
D.~Novikov and S.~Yakovenko.
\newblock Simple exponential estimate for the number of real zeros of complete
  abelian integrals.
\newblock In {\em Annales de l'institut Fourier}, volume 45(4), pages 897--927,
  1995.

\bibitem{O03}
S.~Y. Orevkov.
\newblock Riemann existence theorem and construction of real algebraic curves.
\newblock {\em Ann. Fac. Sci. Toulouse Math. (6)}, 12(4):517--531, 2003.

\bibitem{P22}
G.~P{\'o}lya.
\newblock On the mean-value theorem corresponding to a given linear homogeneous
  differential equations.
\newblock {\em Transactions of the American Mathematical Society},
  24(4):312--324, 1922.

\bibitem{PS76}
G.~Pólya and G.~Szeg{\"o}.
\newblock {\em Problems and Theorems in Analysis}, volume~II.
\newblock Springer-Verlag, 1976.

\bibitem{S11}
F.~Sottile.
\newblock {\em Real solutions to equations from geometry}, volume~57 of {\em
  University Lecture Series}.
\newblock American Mathematical Society, Providence, RI, 2011.

\bibitem{VP75}
M.~Voorhoeve and A.~Van Der~Poorten.
\newblock Wronskian determinants and the zeros of certain functions.
\newblock In {\em Indagationes Mathematicae (Proceedings)}, volume 78(5), pages
  417--424. Elsevier, 1975.

\end{thebibliography}
\def\cprime{$'$}

\end{document}